\documentclass[10pt]{amsart}
\usepackage{amsmath}
\usepackage{amssymb}
\usepackage{enumerate}
\usepackage{amsbsy}
\usepackage{amsfonts}
\usepackage{color}

\newtheorem{thm}{Theorem}[section]
\newtheorem{lem}[thm]{Lemma}
\newtheorem{prop}[thm]{Proposition}

\newtheorem{rem}[thm]{Remark}

\numberwithin{equation}{section}

\newcommand{\h}{\dot{H}^1}
\newcommand{\hr}{\dot{H}_{rad}^1}
\newcommand{\rt}{\mathbb{R}^3}
\newcommand{\les}{\lesssim}
\newcommand{\lam}{\lambda}
\newcommand{\gam}{{\gamma}}
\newcommand{\vp}{\varphi}
\newcommand{\ve}{\varepsilon}

\newcommand{\eg}{E_g}
\newcommand{\gs}{g_s}
\newcommand{\egc}{E_{g,c}}

\newcommand{\s}{S_w(I)}
\newcommand{\sst}{S_w(I^*)}
\newcommand{\yi}{Y_{i,w}(I)}

\newcommand{\ljn}{\lam_{j,n}}

\newcommand{\kjn}{\kappa_{j,n}}
\newcommand{\kon}{\kappa_{1,n}}

\newcommand{\tjn}{t_{j,n}}

\begin{document}

	\title[Energy-critical inhomogeneous NLS]{On the focusing energy-critical inhomogeneous NLS: weighted space approach}

	\author[Y. Cho]{Yonggeun Cho}
	\address{Department of Mathematics, and Institute of Pure and Applied Mathematics, Jeonbuk National University, Jeonju 54896, Republic of Korea}
	\email{changocho@jbnu.ac.kr}

	\author[K. Lee]{Kiyeon Lee}
	\address{Department of Mathematics, Jeonbuk National University, Jeonju 54896, Republic of Korea}
	\email{leeky@jbnu.ac.kr}

	
	\begin{abstract}
In this paper we consider the global well-posedness (GWP) and finite time blowup problem for the 3D focusing energy-critical inhomogeneous NLS with spatial inhomogeneity coefficient $g$ such that $g(x) \sim |x|^{-b}$ for $0 \le b < 2$. The difficulty of this problem comes from the singularity of $g$. In the previous result \cite{chkl3} the authors showed the GWP for $0 \le b < \frac43$ by Kenig-Merle argument based on the standard Strichartz estimates. Here we extend the GWP to the coefficient with more serious singularity, $\frac43 \le b < \frac32$. For this purpose, we improve the local theory and develop a new profile decomposition based on weighted Strichartz estimates.
	\end{abstract}

		\maketitle

\section{Introduction}
We consider the following Cauchy problem for an inhomogeneous nonlinear Schr\"odinger equation:
\begin{eqnarray}\label{maineq}
	\left\{\begin{array}{l}
	iu_t + \Delta u + g|u|^{p-1}u = 0 \;\; \mathrm{in}\;\; \mathbb{R}^{1+3},\\
	u(0) = \varphi \in \h(\rt),
	\end{array} \right.
\end{eqnarray}
where $g \in C^1(\rt\backslash \{0\})$ is the coefficient representing interaction among particles. The equation \eqref{maineq} can be a  model of  dilute Bose-Einstein condensate when the two-body interactions of the condensate are considered. For this see \cite{bpvt, ts}. Also, it has been considered to study the laser guiding in an axially nonuniform plasma channel. See \cite{gill, ss, ts}.

To maintain the $\h$-scaling invariance structure we assume that $p = 5 - 2b$ for $0 \le b < 2$ and
\begin{align}\label{scaling}
0 \le g_i \le |x|^bg(x) \le g_s \;\;  \mbox{and}\;\; |x||\nabla g(x)| \lesssim |x|^{-b} \;\;\mbox{for any} \;\;x \neq 0,
\end{align}
 where $g_i = \inf |x|^bg(x)$ and $ g_s = \sup |x|^bg(x)$.
Here $\h$ denotes the homogenous Sobolev space defined by
$$
\h = \{f \in L_x^6 :  \|f\|_{\h}:=\|\nabla f\|_{L_x^2}<+\infty\}.
$$

	The energy $E_g$ of the solution to \eqref{maineq} is defined by
\begin{eqnarray}\label{energy}
E_g(u(t))&:=& \frac12 \|\nabla u(t)\|_{L_x^2}^{2} - \frac1{p+1}\int g|u(t)|^{p+1} dx.
\end{eqnarray}

\newcommand{\rp}{{ r_0}}
\newcommand{\qp}{{ q_0}}
\newcommand{\gm}{{\gam^*}}

We say that \eqref{maineq} is locally well-posed if there exists a maximal existence time interval $I^*$ such that there exists a unique solution $u \in C(I^*; \h)$ and $u$ depends continuously on the initial data. The local well-posedness (LWP) can be usually shown by a contraction argument based on the Strichartz estimates \cite{cawe, chkl3, leeseo}. In this paper the $L_t^\qp L_x^\rp(|x|^{-\rp\,\gm})$-norm controls our whole contraction argument. Here
\begin{align}\label{adm-pair}
\gm = \frac12-4\ve,\;\; \rp = \frac6{1-6\ve},\;\;\mbox{ and }\;\;\qp = \frac4{1-2\ve}
\end{align}
for arbitrarily small $\ve > 0$. The space $L_t^\qp L_x^\rp(|x|^{-\rp\,\gm})$ is $\dot H^1$-scaling invariance, that is, $\|u_\lambda\|_{L_t^\qp L_x^\rp(|x|^{-\rp\,\gm})} = \|u\|_{L_t^\qp L_x^\rp(|x|^{-\rp\,\gm})}$ for $u_\lambda(t, x) = \lambda^\frac12u(\lambda^2t, \lambda x)$ and for any $\lambda > 0$. The problem \eqref{maineq} is said to be globally well-posed if $I^* = \mathbb R$ and the global solution $u$ is said to scatter in $\h$ if there exists linear solutions $u_\pm$ such that $u \to u_\pm$ in $\h$ as $t \to \pm\infty$. The solution is said to blow up if
$$
\left(\int_{I^*}\left(\int_{\mathbb R^3}\big(|x|^{-\gm}|u(t, x)|\big)^{\rp}\,dx\right)^{\frac{\qp}{\rp}}dt\right)^{\frac1\qp} = +\infty,
$$
We also use the terminology of finite time blowup when $I^*$ is bounded.

The aim of this paper is to establish a global theory for radial solutions: the global well-posedness (GWP), the scattering, and  the finite time blowup to \eqref{maineq}. In the previous paper \cite{chkl3} the authors considered a global theory for $g$ with $0 \le  b < \frac43$ which was shown by a concrete concentration-compactness argument based on the local theory and profile decomposition. The restriction of index $b$ is due to the lack of local theory of \eqref{maineq}. In this paper, we overcome it and extend the range of $b$ up to $\frac32$. In order to handle the $g$ with $\frac43 \le b < \frac32$ we develop an improved local theory, which consists of LWP and long-time perturbation, and develop a new profile decomposition based on the weighted space $L_t^\qp L_x^\rp(|x|^{-\rp\,\gm})$. The global theory can be shown straightforwardly by the concentration-compactness argument of \cite{ken, chkl3, dinh, fagu, mimurzh}. One can find another results about inhomogeneous NLS in \cite{coge, dinh2, guz, merle}.

To state our main results we first introduce a variational condition which restricts the lower and upper bounds of $g$ as follows:
\begin{align}\label{var-con}
g_0 := \gs(3 - b - g_i)  \le 2 - b.
\end{align}
where $g_i, g_s$ are in \eqref{scaling} and a {\it rigidity condition} for $g$ such that
\begin{align}\label{rig-con}
-bg(x) \le x\cdot \nabla g(x) \;\;\mbox{for all}\;\;x \neq 0.
\end{align}
These conditions are crucial for variational estimates and localized virial estimates which play a key major role in the concentration-compactness argument.
The condition \eqref{var-con} seems to be more or less technical. In fact it is necessary while comparing the initial data with the ground state $Q_b$, which is the solution $Q_b(x)= \left(1+ \frac{|x|^{2-b}}{3-b}\right)^{-\frac1{2-b}}$ to the elliptic problem
    \begin{align}\label{ell}
\Delta Q_b + |x|^{-b}Q_b^{5-2b} = 0.
\end{align}
For this see Remark 2.1 of \cite{ya} and Appendix of \cite{chkl3}. The condition \eqref{rig-con} enables us to control the error term occurring when we deal with the lower bound for the second derivative of localized virial quantity.


Now we are ready to state our the main result.
\begin{thm}\label{mainresult}
	Let $\frac43 \le  b < \frac32$. Let $g$ be a radial function satisfying \eqref{scaling}, \eqref{var-con}, and \eqref{rig-con}.		
	Suppose that $\varphi \in \hr := \{f \in \h : f\;\;\mbox{is radial}\;\}$,
	\begin{align}\label{e-cond}
	E_g(\varphi) < E_g(Q_b),\;\;\mbox{and}\;\; \gs\|\varphi\|_{\h}^2 < \|Q_b\|_{\h}^2.
	\end{align}
	Then \eqref{maineq} is globally well-posed in $\hr$ and the solution $u$ scatters in $\hr$.
\end{thm}
The upper bound $\frac32$ of $b$ is required to control, by the weighted norm $L_t^\qp L_x^\rp(|x|^{-\rp\gm})$, the nonlinear terms appearing while dealing with LWP. This weighted space argument is inevitable in our local theory for the present.(See the Remark \ref{maincondi}.) The gap $\frac32 \le b < 2$ will be hopefully filled in the near future. Since the local theory is based on weighted spaces, we need to develop a new profile decomposition associated with the weighted space. Once the local theory and profile decomposition are established, one can readily prove Theorem \ref{mainresult} by following the concentration-compactness argument of \cite{km, ken, chkl2, chkl3}. Hence we focus mainly on the local theory and profile decomposition and sketch the concentration-compactness argument in this paper to avoid the duplication.


Let us now consider a blowup result. Our blowup result follows from the standard virial argument for which we need to control the upper bound of the second derivative of localized virial quantity. This can be done by assuming that
\begin{align}\label{virial-con}
x\cdot \nabla g(x) \le (6 - b)(k_g-\rho)g(x) \;\;\mbox{for all}\;\;x \neq 0,
\end{align}
where $k_g = \frac{2 - b - g_0}{3 - b - g_0}$ and for some $\rho \ge 0$.
Then we get the sharp blowup result as follows.
\begin{thm}\label{blowup-thm}
	Let $\frac43 \le b < \frac32$. Let $g$ be a nonnegative and bounded function satisfying \eqref{scaling}, \eqref{var-con}, and \eqref{virial-con}.		
	\begin{enumerate}
		\item[$(1)$] Suppose that $\varphi \in \h$, $|x|\varphi \in L^2$,
		\begin{align}\label{b-cond}
		E_g(\varphi) < E_g(Q_b),\;\;\mbox{and}\;\; \gs\|\varphi\|_{\h}^2 \ge \|Q_b\|_{\h}^2.
		\end{align}
		Then the solution $u$ to \eqref{maineq} blows up in finite time.
		\item[$(2)$] Suppose that $g$ is radial, $\rho > 0$, and $\varphi \in \hr$ satisfies \eqref{b-cond}.
		Then the radial solution $u$ to \eqref{maineq} blows up in finite time.
	\end{enumerate}
\end{thm}
Note that the radial symmetry is not necessary for the proof of $(1)$. In $(2)$ the moment condition $|x|\varphi \in L_x^2$ has been replaced with the radial symmetry and $\vp \in L_x^2$. This is due to the space-decay estimate of Strauss \cite{stra}. The condition $\rho > 0$ in $(2)$ is required to handle error terms appearing in localized virial argument that is not necessary for $(1)$.

\noindent\textbf{Notations.}\\
\noindent$\bullet$ Mixed-normed spaces: For a Banach space $X$ and an interval $I$, $u \in L_I^q X$ iff $u(t) \in X$ for a.e. $t \in I$ and $\|u\|_{L_I^qX} := \|\|u(t)\|_X\|_{L_I^q} < \infty$. Especially, we denote  $L_I^qL_x^r = L_t^q(I; L_x^r(\rt))$, $L_{I, x}^q = L_I^qL_x^q$, $L_t^qL_x^r = L_{\mathbb R}^qL_x^r$.

\noindent$\bullet$ Weighted spaces: For $\gam \ge 0$ and $I$ an interval, $u \in L_I^qL_x^r(|x|^{-r\gam}) $ iff $\|u\|_{L_I^qL_x^r(|x|^{-r\gam})} := \||x|^{-\gam}u\|_{L_I^qL_x^r} < \infty$.

\noindent$\bullet$ As usual different positive constants depending only on $b, g_i, g_s$ are denoted by the same letter $C$, if not specified. $A \lesssim B$ and $A \gtrsim B$ means that $A \le CB$ and
$A \ge C^{-1}B$, respectively for some $C>0$. $A \sim B$ means that $A \lesssim B$ and $A \gtrsim B$.\\


\section{Local theory}
In this section, we deal with a local theory on \eqref{maineq}, which consists of local well-posedness (LWP) and long-time perturbation.

\subsection{Preliminaries}
We first introduce some preliminaries which are useful both in local and global theories.
By Duhamel's principle the equation \eqref{maineq} is rewritten as the integral equation:
\begin{eqnarray}\label{inteq}
u = e^{it\Delta}\varphi + i \int_{0}^{t} e^{i(t-t')\Delta}g|u(t')|^{p-1}u(t')dt'.
\end{eqnarray}
Here we define the linear propagator $e^{it\Delta}$ the solution to the linear problem $i\partial_tv=-\Delta v$ with initial data $v(0)=f$. It is formally given by
$$e^{it\Delta}f = \mathcal{F}^{-1}\left(e^{-it|\xi|^2}\mathcal{F}(f)\right)= (2\pi)^{-3}\int_{\rt} e^{i(x\cdot\xi - t|\xi|^2)}\widehat{f}(\xi)d\xi,$$
where $\widehat{f} = \mathcal F( f)$ denotes the Fourier transform of $f$ and $\mathcal F^{-1}(h)$ the inverse Fourier transform of $h$ such that
$$\mathcal{F}(f)(\xi) = \int_{\rt} e^{- ix\cdot \xi} f(x)\,dx,\quad \mathcal F^{-1} (h)(x) = (2\pi)^{-3}\int_{\rt} e^{ix\cdot \xi} h(\xi)\,d\xi.$$

Recently, a weighted version of Strichartz estimate was considered in \cite{leeseo}. It can be described as follows. Let $2 \le q,r \le \infty$ and $0 \le \gam < 1$. If a pair $(q,r)$ satisfies the equation $\frac2q = 3(\frac12 - \frac1r) + \gam$, then we call it $\gam$-{\it admissible} pair. If $\gam = 0$, then it is called just {\it admissible}.
\begin{lem}\cite{leeseo,kt}\label{stri-w}
	Let $(q,r)$ be $\gamma$-{\it admissible} and $(\tilde{q},\tilde{r})$ be $\tilde{\gamma}$-{\it admissible}. Then we have
	\begin{align*}
	\|e^{it\Delta}\vp\|_{L_t^qL_x^r(|x|^{-r\gamma})} &\les \|\vp\|_{L_x^2},\\
	\left\|\int e^{i(t-t')\Delta}F\,dt'\right\|_{L_t^qL_x^r(|x|^{-r\gamma})} &\les \|F\|_{L_t^{\tilde{q}'}L_x^{\tilde{r}'}(|x|^{\tilde{r}\tilde{\gamma}})}.\\
	\end{align*}
\end{lem}
Note that the pair $(q_0, r_1)$ with $\frac1{r_1} =  \frac1{r_0} + \frac13$ is a $\gm$-admissible pair, where $q_0, r_0, \gm$ is defined as \eqref{adm-pair}. Every weighted Strichartz pair satisfies the $\dot H^1$-scaling invariance, that is, $\|\nabla u_\lambda\|_{L_t^qL_x^r(|x|^{-r\gam})} = \|\nabla u\|_{L_t^qL_x^r(|x|^{-r\gam})}$ for any $\gam$-admissible pair $(q, r)$.

Now fix $\frac43 \le b < \frac32$ and set $p = 5 - 2b$. For a small $0< \ve < \frac38(p-2)$ we define numbers associated with the weighted Strichartz estimate such that
$$
\theta = 2p-\frac{10}{3} -4p\ve, \;\; \gm = \frac12 -4\ve,
$$
\begin{align*}
\frac1{q^*}&= \frac{p}4-\frac{\theta}4+\left(\frac{\theta}2 -\frac23\right)\ve = \frac5{6} -\frac{p}4 +\left(2p -\frac{7}3\right)\ve -2p\ve^2,
\end{align*}
and
\begin{align*}
\frac1{r^*}&= \frac{10}{9} -\frac{p}{6}-\frac{\theta}6 + \left(\theta+\frac49\right)\ve=\frac53 -\frac{p}2 +\left(\frac{8}3p-\frac{26}{9}\right)\ve -4p\ve^2
\end{align*}
 Then $\bigg((p-\theta)q^*,\;(p-\theta)r^*\bigg)$ is $\gm$-{\it admissible}.

Next we introduce a Hardy-Littlewood-Sobolev inequality of weighted type.
\begin{lem}\cite{sw}\label{emb-w}
	Let $f \in W^{1, r}(|x|^{r\beta})$ for $1 < r < \infty$ and $\frac13 < \frac{\beta}3 + \frac1r < 1$. If $\alpha \le \beta$, $\frac1q = \frac1r - \frac{1+\alpha-\beta}3$, and $r \le q < \infty$, then we have
	$$
	\||x|^{\alpha}f\|_{L_x^q} \le C \||x|^{\beta}|\nabla| f\|_{L_x^r}.
	$$
\end{lem}
Since $ 0 < (\frac{14}{3} -2p)\ve +4p\ve^2$ and hence $\frac13 < -\frac{\gm}3 +\frac1{(p-\theta)r^*} < 1$, one can
apply Lemma \ref{emb-w} with $\alpha = -\gm - 1$, $\beta = -\gm$, $r = q = (p - \theta)\gm$ as follows:
\begin{align}\label{em-lwp}
\||x|^{-\gm-1}u\|_{L_x^{(p-\theta)r^*}} \le C \||x|^{-\gm}\nabla u\|_{L_x^{(p-\theta)r^*}}.
\end{align}

The following lemma is on the nonlinear estimate.
\begin{lem}\label{wei-esti}
Let $\theta, \gm$ be as above and $\frac43\le b < \frac32$. Then there exists $(\widetilde{q},\widetilde{r}),\; (q_2,r_2)$, and $\widetilde{\gam}$ such that $(\widetilde{q},\widetilde{r})$ is $\widetilde{\gam}$-{\it admissible}, $(q_2, r_2)$ is $\gm$-{\it admissible}, and
$$
\|\nabla (g|u|^{p-1}u)\|_{L_t^{\widetilde{q}'}L_x^{\widetilde{r}'}(|x|^{\widetilde{\gam}\widetilde{r}'})} \le C\|u\|_{L_t^{q_0}L_x^{r_0}(|x|^{-\gm r_0})}^{\theta}\|\nabla u\|_{L_t^{q_2}L_x^{r_2}(|x|^{-\gm r_2})}^{p-\theta}
$$
\end{lem}

\begin{proof}
Let $\frac1{\widetilde{q}} = 1-\frac{p}4+\frac23\ve,\; \frac1{\widetilde{r}}=\frac{p}{6} -\frac19-\frac49\ve$ for $0 < \ve < \frac38(p -2)$. Then $(\widetilde{q},\widetilde{r})$ is $\frac16$-{\it admissible}. By the scaling condition \eqref{scaling} and Lemma \ref{emb-w} we obtain
\begin{align}
\begin{aligned}\label{w-esti}
\||x|^{-b-1}|u|^{p-1}u\|_{L_t^{\widetilde{q}'}L_x^{\widetilde{r}'}(|x|^{\frac16\widetilde{r}'})} &=  \||x|^{-b-1+\frac16}|u|^{p-1}u\|_{L_t^{\widetilde{q}'}L_x^{\widetilde{r}'}}\\
&\le C \||x|^{-\gm}u\|_{L_t^{q_0}L_x^{r_0}}^{\theta}\||x|^{-\gm-1}u\|_{L_t^{q_2}L_x^{r_2}}^{p-\theta}\\
&\le C \||x|^{-\gm}u\|_{L_t^{q_0}L_x^{r_0}}^{\theta}\||x|^{-\gm}\nabla u\|_{L_t^{q_2}L_x^{r_2}}^{p-\theta}
\end{aligned}\end{align}
and
\begin{align}
\begin{aligned}\label{w-esti2}
\||x|^{-b}|u|^{p-1}\nabla u\|_{L_t^{\widetilde{q}'}L_x^{\widetilde{r}'}(|x|^{\frac16\widetilde{r}'})} &=  \||x|^{-b+\frac16}|u|^{p-1}\nabla u\|_{L_t^{\widetilde{q}'}L_x^{\widetilde{r}'}}\\
&\le C \||x|^{-\gm}u\|_{L_t^{q_0}L_x^{r_0}}^{\theta}\||x|^{-\gm-1}u\|_{L_t^{q_2}L_x^{r_2}}^{p-1-\theta}\||x|^{-\gm}\nabla u\|_{L_t^{q_2}L_x^{r_2}}\\
&\le C \||x|^{-\gm}u\|_{L_t^{q_0}L_x^{r_0}}^{\theta}\||x|^{-\gm}\nabla u\|_{L_t^{q_2}L_x^{r_2}}^{p-\theta}.
\end{aligned}\end{align}
where $\frac1{q_0}=\frac14 -\frac{\ve}{2},\; \frac1{r_0}=\frac16 - \ve$ and $q_2=(p-\theta)q^*,\;r_2 = (p-\theta)r^*$ with $(q^*,r^*)$ is as in above.
Here we used the H\"older pairs such that
\begin{align*}
\frac1{\widetilde{q}'}=\frac{p}4 - \frac 23 \ve & =  \theta\left(\frac14 -\frac{\ve}2\right) + \frac{p}4-\frac{\theta}4+\left(\frac{\theta}2 -\frac23\right)\ve = \frac{\theta}{q_0} +  \frac{p-\theta}{q_2}\\
\frac1{\widetilde{r}'}=\frac89 -\frac{p}6 + \frac49\ve & = \theta\left(\frac16 -\ve \right) + \frac{10}{9} -\frac{p}{6}-\frac{\theta}6 + \left(\theta+\frac94\right)\ve = \frac{\theta}{r_0} + \frac{p-\theta}{r_2}.
\end{align*}
This completes the proof of Lemma \ref{wei-esti}.
\end{proof}

\begin{rem}\label{maincondi}
 The condition $0 < (\frac{14}{3} -2p)\ve + 4p\ve^2$ is essential for the inequality \eqref{em-lwp}. Also, the condition $0 < \ve < \frac38(p -2)$ is necessary for $(\widetilde{q},\widetilde{r})$ to be $\frac16$-{\it admissible}. These constraints say that $p$ must satisfy that $2 < p \le \frac73$, that is, $\frac43 \le b < \frac32$.
\end{rem}

From now on we denote
$$
S_w(I)=L_I^{q_0}L_x^{r_0}(|x|^{-r_0\gm}),\;\; Y_{1,w}(I)=L_I^{q_0}L_x^{r_1}(|x|^{-r_0\gm}), \;\;\mbox{and}\;\;  Y_{2,w}(I)=L_I^{q_2}L_x^{r_2}(|x|^{-r_2\gm}),
$$
where $q_0,\;q_2,\; r_0,\;r_1,$ and $r_2 $ are same as stated above.

The local well-posedness is shown in \cite{leeseo}. However, we need a different LWP result adapted to concentration-compactness argument. We now state our LWP result.
\begin{prop}\label{lwp}
	Let $\varphi\in\h, 0\in I$ an interval, and $\frac43 \le b <\frac32$. Assume that $\|\varphi\|_{\h}\le A$. Then there exists $\delta = \delta(A)$ satisfied following:	
	If $\|e^{it\Delta}\varphi\|_{S_w(I)} < \delta$, then there exists a unique solution $u$ of \eqref{maineq} in $I \times \rt$ with $u\in C(I;\h(\rt))$,
	$$
	\|u\|_{S_w(I)}\le 2\delta,\;\;\mbox{and} \;\; \|\nabla u\|_{\yi}<\infty\;\;(i=1,2).
	$$
	
	In particular, if $\varphi_k \to \varphi$ in $\h$, then the corresponding solutions $u_k \to u$ in $C(I; \h)$ as $k \to \infty$.
\end{prop}

\begin{proof} We use the contraction mapping principle. To this end we fix $r, s > 0$, to be chosen later. Let us define a complete metric space $(\mathcal{M}_{r, s}, d)$ and a mapping $\Phi$ as following:
	\begin{align*}
	&\mathcal{M}_{r, s} = \{v \in C(I; \h) : \|v\|_{L_I^\infty \h} \le 2A, \;\|v\|_{\s}\le r,\; \|\nabla v\|_{\yi}\le s \;(i=1,2)\},\\
	&d(u, v) = \|v - v'\|_{L_I^\infty \h} + \|v-v'\|_{\s} + \sum_{i = 1,2}\|\nabla(v-v')\|_{\yi},\\
	&\Phi(v) = e^{it\Delta}\varphi + i \int_{0}^{t} e^{i(t-t')\Delta}f(v)dt',\quad f(v) = g|v|^{p-1}v, \quad p = 5-2b.
	\end{align*}
	By Lemma \ref{wei-esti} we obtain for each $i=1,\, 2$ that
	\begin{align}
	\begin{aligned}\label{L10}
	\|\nabla\Phi(v)\|_{\yi} &\le C\|\nabla \varphi\|_{L_x^2} + C\|\nabla f(v)\|_{L_I^{\widetilde{q}'}L_x^{\widetilde{r}'}(|x|^{\frac16\widetilde{r}'})}\\
	& \le CA+C\|v\|_{\s}^{\theta}\|\nabla v\|_{Y_{2,w}(I)}^{p-\theta}\\
	& \le C(A+r^{\theta}s^{p-\theta}).
	\end{aligned}\end{align}
	Now we take $r,s$ satisfied $s\le 2AC$ and $ Cp^{\theta}s^{p-\theta} \le \min(\frac12,\frac1{2C}) $. Then we get $\|\nabla\Phi(v)\|_{\yi} \le s$ for each $i=1,\,2$ and
	$$
	\|\Phi(v)\|_{L_t^{\infty}\h}\le A +Cr^{\theta}s^{p-\theta} \le 2A.
	$$
	
	By weighted-type Sobolev embedding (Lemma \ref{emb-w}) we may have
$$
\|\Phi(v)\|_{\s} \le \|e^{it\Delta}\vp\|_{\s} +C\left\|\nabla\int_{0}^{t} e^{i(t-t')\Delta}f(v)dt'\right\|_{Y_{1,w}(I)} \le \delta + Cr^{\theta}s^{p-\theta}.
$$
Hence $\|\Phi(v)\|_{\s} \le 2\delta$ for $r=2\delta$.
The above estimates yield that $\Phi$ is self-mapping on $\mathcal{M}_{r,s}$.

We next show $\Phi$ is contraction on $\mathcal{M}_{r,s}$.
	\begin{align*}
d(\Phi(v), \Phi(v')) &\le C\|(|v|^{p-1} + |v'|^{p-1})|x|^{-b-1}|v - v'|\|_{L_I^{\widetilde{q}'}L_x^{\widetilde{r}'}(|x|^{\frac16\widetilde{r}'})} + C\||x|^{-b}|v|^{p-1}|\nabla v - \nabla v'|\|_{L_I^{\widetilde{q}'}L_x^{\widetilde{r}'}(|x|^{\frac16\widetilde{r}'})}\\
&\qquad + C\||x|^{-b} |\nabla v'|(|v|^{p-2} + |v'|^{p-2})|v-v'|\|_{L_I^{\widetilde{q}'}L_x^{\widetilde{r}'}(|x|^{\frac16\widetilde{r}'})}\\
& \le C(\|v\|_{\s}^{\theta} + \|v'\|_{\s}^{\theta})\||x|^{-1}(v-v')\|_{Y_{2,w}(I)}^{p-\theta} +C \|v\|_{\s}^{\theta}\||x|^{-1}v\|_{Y_{2,w}(I)}^{p-1-\theta}\|\nabla(v-v')\|_{Y_{2,w}(I)} \\
&\qquad  +C\|v'\|_{Y_{2,w}(I)}(\|v\|_{\s}^{p-2} + \|v'\|_{\s}^{p-2})\|v-v'\|_{\s}^{\theta-p+2}\||x|^{-1}(v-v')\|_{Y_{2,w}(I)}^{p-1-\theta}
\end{align*}
This follow that
$$
d(\Phi(v), \Phi(v')) \le C r^{\theta}s^{p-1-\theta} d(v,v')
$$
Hence $\Phi$ is contraction on $\mathcal{M}_{r,s}$ for $Cr^{\theta}s^{p-1-\theta}<1$.

	The continuous dependence on initial data follows immediately from the above contraction argument. This completes the proof of Proposition \ref{lwp}.
\end{proof}

\begin{rem}
	\begin{enumerate}
		\item [$(i)$]\label{blowup-cri} (Blowup criterion)
		Proposition \ref{lwp} implies the existence of maximal existence time interval $I^*$. Moreover, one can immediately deduce the blowup criterion: if $\|u\|_{\sst} < +\infty$, then $I^* = \mathbb R$, and if $I^*$ is bounded, then $\|u\|_{\sst} = +\infty$. We also conclude that if $\|\varphi\|_{\h}$ is sufficiently small, then $I^* = \mathbb R$.
		\item [$(ii)$]\label{scattering}($\h$ scattering)
		Suppose that $I^* = \mathbb R$ and $\|u\|_{\sst} < +\infty$. Let us set
		$$
		\varphi_\pm := \varphi + i\int_0^{\pm \infty} e^{-it'\Delta}[ g|u|^{p-1}u]\,dt'.
		$$
		Then the solution $u$ scatters to $e^{it\Delta}\varphi_\pm$ in $\h$ by standard duality argument.
		\item [$(iii)$] \label{m-e-cons}(Mass-energy conservation)
		Let us define the mass by $\|u(t)\|_{L_x^2}^2$ for the solution $u$ to \eqref{maineq}. Then we can readily get the mass conservation for initial data $\vp \in H^1$. Also, energy conservation is established for the same initial data.
	\end{enumerate}
	
\end{rem}


\begin{prop}[Long-time perturbation]\label{per}
	Let $g$ be a radial function satisfying \eqref{scaling} with $\frac43 \le b < \frac32$. Let $I\subset \mathbb{R}$ be a time interval containing $0$ and $\widetilde{u}$ be a radial function defined on $I \times \rt$. Assume that $\widetilde{u}$ satisfies following:
	$$
	\|\widetilde{u}\|_{L_t^{\infty}\h} \le A \;\;\mbox{and}\;\; \|\widetilde{u}\|_{\s}\le M
	$$
	for some constants $M, A > 0$ and
	$$
	i\partial_t\widetilde{u} + \Delta\widetilde{u} + f(\widetilde{u}) = e \;\;\mbox{for}\;\; (t,x)\in I\times \rt,
	$$
	where $f(\widetilde u)=g|\widetilde u|^{p-1}\widetilde u$ and that
	$$
	\|\vp-\widetilde{u}(0)\|_{\h} \le A',\; \|\nabla e\|_{L_I^{\widetilde{q}'}L_x^{\widetilde{r}'}(|x|^{\frac16\widetilde{r}'})}   \le \ve,\;\;\mbox{and}\;\; \|\nabla e^{it\Delta}[\varphi-\widetilde{u}(0)]\|_{\yi}\le \varepsilon \;\;(i=1,2)
	$$
	for $\widetilde{q},\widetilde{r},\widetilde{\gam}$ are as in the proof of Proposition \ref{lwp}. Then there exists $\varepsilon_0 = \varepsilon_0(M,A,A')$ and a unique solution $u \in C(I; \h_{rad})$ with $u(0)=\varphi$ in $I\times \mathbb{R}$, such that for $0 < \varepsilon < \varepsilon_0$ with
	$$
	\|u\|_{\s} \le C(M,A,A')\;\;\mbox{and}\;\; \|u(t)-\widetilde{u}(t)\|_{\h} \le A' + C(M,A,A')\varepsilon\;\;\mbox{for all}\;\; t \in I.
	$$

\end{prop}

\begin{proof}
Without loss of generality, we assume that $I =\left[0, a \right)$ for some $0 < a \le \infty$. By H\"older and Lemma \ref{emb-w} we get
$$
\|\nabla (g|u|^{p-1}u) \|_{L_I^{\widetilde{q}'}L_x^{\widetilde{r}'}(|x|^{\frac16\widetilde{r}'})} \les \|u\|_{\s}^{\theta}\|\nabla u\|_{Y_{2,w}(I)}^{p-\theta}.
$$
By the integral equation \eqref{inteq} for $\widetilde{u}$ and Proposition \ref{lwp} we also get
$$
\|\nabla \widetilde{u}\|_{Y_{i,w}(I_k)} \le CA + \rho \|\nabla\widetilde{u}\|_{Y_{i,w}(I_k)}^{p-\theta}
$$
for $\{I_k\}$ satisfying $\bigcup I_k = I$ and $C\|\widetilde{u}\|_{S_w(I_k)}^{\theta} \le \rho$. The continuous argument yields $ \|\nabla \widetilde{u}\|_{Y_{i,w}(I_k)} < 2CA$ for sufficiently small $\rho$ and hence one can readily obtain
$$
\|\nabla \widetilde{u}\|_{Y_{i,w}(I_k)} \le \widetilde{M}
$$
for some $\widetilde{M}=\widetilde{M}(M,A)$.

Next we write $u = \widetilde{u} +w$, so that the equation for $w$ is written as
\begin{eqnarray*}
	\left\{\begin{array}{l}
		iw_t + \Delta w = -f(\widetilde{u} +w) + f(\widetilde{u}) + e, \\
		w(0) = \varphi - \widetilde{u}(0).
	\end{array} \right.
\end{eqnarray*}
Then for arbitrary $\eta > 0$, there exists $I_j= \left[a_j,a_{j+1}\right)$ such that $\bigcup_{j=1}^JI_j = I$ and $\|\nabla \widetilde{u}\|_{Y_{i,w}(I_j)}\le \eta \;(i=1,2)$. On $I_j$ $w$ satisfies
$$
w(t) = e^{i(t-a_j)\Delta}w(a_j) + i \int_{a_j}^{t}e^{i(t-t')\Delta}(f(\widetilde{u} +w) -f(\widetilde{u}))dt' - i\int_{a_j}^{t}e^{i(t-t')\Delta}e(t')dt'.
$$
By Sobolev embedding and Lemma \ref{emb-w}, we get
\begin{align*}
\sum_{i=1}^{2}\|\nabla w\|_{Y_{i,w}(I_j)} &\le \left(\sum_{i=1}^{2}\|\nabla e^{i(t-a_j)\Delta}w(a_j)\|_{Y_{i,w}(I_j)}+2C\varepsilon\right) + C\eta^{p-1}\sum_{i=1}^{2}\|\nabla w\|_{Y_{i,w}(I_j)}\\
& \qquad\quad   +C\left(\sum_{i=1}^{2}\|\nabla w\|_{Y_{i,w}(I_j)}\right)^p.
\end{align*}
Thus, if $C\eta^{p-1} \le \frac13$, we have	
$$
\sum_{i=1}^{2}\|\nabla w\|_{Y_{i,w}(I_j)} \le \frac32\eta_j + C\left(\sum_{i=1}^{2}\|\nabla w\|_{Y_{i,w}(I_j)}\right)^p,
$$
where $\eta_j = \sum_{i=1}^{2}\|\nabla e^{i(t-a_j)\Delta}w(a_j)\|_{Y_{i,w}(I_j)}+ 2C\varepsilon$.

From the standard continuity argument, we can find $C_0 > 0$ satisfying that
$$
\sum_{i=1}^{2}\|\nabla w\|_{Y_{i,w}(I_j)} \le 3 \eta_j  \;\;\;\mbox{and}\;\;\;   C \left(\sum_{i=1}^{2}\|\nabla w\|_{Y_{i,w}(I_j)}\right)^p \le 3\eta_j,
$$
provided $\eta_j \le C_0$. Repeating the above argument for the equation
\begin{align*}
e^{i(t-a_{j+1})\Delta}w(a_{j+1}) = e^{i(t-a_j)\Delta}w(a_j) + i \int_{a_j}^{a_{j+1}}e^{i(t-t')\Delta}(f(\widetilde{u} +w) -f(\widetilde{u}))dt'\\
- i\int_{a_j}^{a_{j+1}}e^{i(t-t')\Delta}e(t')dt',
\end{align*}
we get
\begin{align*}
\sum_{i=1}^{2}\|e^{i(t-a_{j+1})\Delta}w(a_{j+1})\|_{Y_{i,w}(I_{j+1})} &\le \sum_{i=1}^{2}\|\nabla e^{i(t-a_j)\Delta}w(a_j)\|_{Y_{i,w}(I_{j+1})} + C\varepsilon\\
&\qquad \qquad+ C\eta^2\sum_{i=1}^{2}\|\nabla w\|_{Y_{i,w}(I_{j+1})} + C\left(\sum_{i=1}^{2}\|\nabla w\|_{Y_{i,w}(I_{j+1})}\right)^p.
\end{align*}
Taking a sufficiently small $\eta$ to satisfy $\eta_{j+1} \le 10 \eta_j$ provided $\eta_j \le C_0$. This always happens if $C10^{J}\varepsilon_0 < C_0$. With this $\varepsilon_0$ we have that for any $0 < \varepsilon < \varepsilon_0$
$$
\|w\|_{\s} + \|\nabla w\|_{Y_{1,w}(I)} + \|\nabla w\|_{Y_{2,w}(I)} \le 3C\sum_{j = 1}^J\eta_j \le \frac{C}3(10^{J+1}-1)\varepsilon.
$$
Hence by setting $C(M, A, A') = C(10^{J+1}-1)\varepsilon_0/3$ we obtain
$$
\|u\|_{\s} \le \|w\|_{\s} + \|\widetilde u\|_{\s} \le C(M, A, A').
$$

Using the Strichartz estimate and Hardy-Sobolev inequality(Lemma \ref{emb-w}) once more, we reach that
$$
\|w\|_{L_I^\infty \h} \le A' + C\varepsilon + C\sum_{j = 1}^J\|\nabla(f(\widetilde u + w) - f(\widetilde u))\|_{L_I^{\widetilde{q}'}L_x^{\widetilde{r}'}(|x|^{\frac16\widetilde{r}'})} \le A' + C(M, A, A')\varepsilon.
$$
\end{proof}


\section{Profile decomposition}
In this section, we provide a new profile decomposition associated with weighted space. Since the concerned data are radially symmetric, we do not consider a general profile decomposition. Instead, we develop a decomposition adapted to radial data and follow the strategy of proof  as in \cite{chhwkwle}.  
\begin{thm}\label{profile-thm}
	Let $\{u_{0, n}\} \subset \hr$ with $\|u_{0, n}\|_{\h}\le A$.
Then up to a subsequence $(\mbox{still called}\; \{u_{0, n}\})$ for any $J \ge 1$ there exists a sequence $\{U_{0,j}\}_{1 \le j \le J}$ and $W_n^J$ in $\hr$ and a family of parameters $(\ljn,\tjn)\in \mathbb{R}^+\times\mathbb{R}$ with
	$$
	\frac{\ljn}{\lam_{j',n}} + \frac{\lam_{j',n}}{\ljn} +\frac{|\tjn-t_{j',n}|}{\ljn^2} \xrightarrow{n\to\infty} \infty \quad  j\neq j'
	$$
	such that
	\begin{enumerate}
		\item[$(i)$]  $u_{0,n}=\sum_{j=1}^{J} \ljn^{-\frac12}U_j^l\left(-\frac{\tjn}{\ljn^2}, \frac{x}{\ljn}\right) + W_n^J, \qquad U_j^l(t, x) := [e^{it\Delta}U_{0, j}](x)$ (linear solution),
		\item[$(ii)$] $\lim\limits_{J\to \infty}\underset{n\to\infty}{\limsup}\left\|\nabla e^{it\Delta}W_n^J\right\|_{L_t^qL_x^r(|x|^{-r\gam})} = 0$ for any $\gam$-admissible pair $(q, r)$ with $2 < q < \infty$,
		\item[$(iii)$]  $\|u_{0,n}\|_{\h}^2 = \sum_{j=1}^{J}\|U_{0, j}\|_{\h}^2 + \|W_n^J\|_{\h}^2 + o(1)$ as $n \to \infty$,
		\item[$(iv)$]   $E_g(u_{0,n}) = \sum_{j=1}^{J}E_g\left(\ljn^{-\frac12}U_j^l(-\frac{\tjn}{\ljn^2}, \frac{\cdot}{\ljn})\right) + E_g(W_n^J)+o(1)$ as $n \to \infty$.
        \item[$(v)$] If $\|\nabla e^{it\Delta}u_{0, n}\|_{L_t^qL_x^r(|x|^{-r\gam})} \ge \delta_0$ for some $\gam$-admissible pair $(q, r)$ and positive $\delta_0$, then there exist $J_0 \ge 1$ and $\alpha  = \alpha(A, \delta_0, J_0, q, r, \gam) > 0$ such that $\|U_{0, 1}\|_{\hr} \ge \alpha$.
        \end{enumerate}
\end{thm}
The proof of energy decomposition $(iv)$ is not involved with weighted space and it was given in \cite{chkl3}. We omit its proof. Suppose that $(i)-(iii)$ of Theorem \ref{profile-thm} have been shown. Then $(v)$ can be shown as follows (also see \cite{chhwoz}).
\begin{proof}[Proof of $(v)$]
From $(i)$ it follows that
$$
e^{it\Delta}u_{0,n}=\sum_{j=1}^{J} \ljn^{-\frac12} e^{i\left(t-\frac{\tjn}{\ljn^2}\right)\Delta}\left[U_{0,j}\left( \frac{\cdot}{\ljn}\right)\right](x) + e^{it\Delta}W_n^J.
$$
By Lemma \ref{limsup} below we have
$$
\lim_{n\to \infty} \left\| \nabla \left( \sum_{j=1}^{J} \ljn^{-\frac12} e^{i\left(t-\frac{\tjn}{\ljn^2}\right)\Delta}\left[U_{0,j}\left( \frac{\cdot}{\ljn}\right)\right](x) \right) \right\|_{L_t^qL_x^r(|x|^{-r\gam})}^4 \le \sum_{j=1}^J \|\nabla  e^{it\Delta} U_{0,j}\|_{L_t^qL_x^r(|x|^{-r\gam})}^4.
$$
Also, from $(ii)$ we obtain
\begin{align*}
&\limsup_{n\to\infty} \left\| \nabla e^{it\Delta}u_{0,n}   -  \nabla\left( \sum_{j=1}^{J} \ljn^{-\frac12} e^{i\left(t-\frac{\tjn}{\ljn^2}\right)\Delta}\left[U_{0,j}\left( \frac{\cdot}{\ljn}\right)\right](x) \right)  \right\|_{L_t^qL_x^r(|x|^{-r\gam})} \\
&\qquad \qquad\qquad= \limsup_{n\to\infty} \|\nabla e^{it\Delta} W_n^J\|_{L_t^qL_x^r(|x|^{-r\gam})} \xrightarrow{J \to \infty} 0.
\end{align*}
This yields that
$$
\delta_0^4 \le \limsup_{n\to\infty} \|\nabla e^{it\Delta}u_{0,n}\|_{L_t^qL_x^r(|x|^{-r\gam})}^4 \le \lim_{J\to \infty}  \sum_{j=1}^J \|\nabla e^{it\Delta} U_{0,j}\|_{L_t^qL_x^r(|x|^{-r\gam})}^4
$$
Thus there exists $J_0 \ge1$ such that $\frac{\delta_0^4}{2} \le \sum_{j=1}^J \|\nabla e^{it\Delta} U_{0,j}\|_{L_t^qL_x^r(|x|^{-r\gam})}^4$ for all $J \ge J_0$.
Therefore, using Strichartz estimates and $(iii)$, we get
\begin{align*}
\sum_{j=1}^J \| \nabla e^{it\Delta} U_{0,j}\|_{L_t^qL_x^r(|x|^{-r\gam})}^4 &\le \left( \sup_{1\le j \le J} \| \nabla e^{it\Delta}U_{0,j}\|_{L_t^qL_x^r(|x|^{-r\gam})}\right)^2 \sum_{j=1}^J \|\nabla e^{it\Delta} U_{0,j}\|_{L_t^qL_x^r(|x|^{-r\gam})}^2\\
&\le C \left( \sup_{1\le j \le J} \|U_{0,j}\|_{\h}\right)^2 \sum_{j=1}^J \|  U_{0,j}\|_{\h}^2\\
&\le C \left( \sup_{1\le j \le J} \|U_{0,j}\|_{\h}\right)^2 \limsup_{n\to\infty}\|u_{0,n}\|_{\h}^2\\
&\le CA^2 \left( \sup_{1\le j \le J} \|U_{0,j}\|_{\h}\right)^2
\end{align*}
Hence, $ \frac{\delta_0^4}{2CA^2} < \|U_{0,j_0}\|_{\h}^2$ for some $j_0 \in \{1, \cdots, J\}$. Relabeling $U_{0,j_0}$ to $U_{0,1}$, we get the desired result.
\end{proof}

The proof of $(i)-(iii)$ of Theorem \ref{profile-thm} can be immediately reduced to the one of $L^2$-version profile decomposition, Proposition \ref{profile} below, by defining $v_{0, n} = |\nabla|u_{0, n}$, $U_{0,j} = |\nabla|^{-1}V_{0,j}$, and $W_n^J = |\nabla|^{-1}w_n^J$.
\begin{prop}\label{profile}
	Let $\{v_{0, n}\} \subset L_{rad}^2$ with $\|v_{0, n}\|_{L_x^2} \le A$. Then up to a subsequence $(\mbox{still called}\; \{v_{0, n}\})$ for any $J \ge 1$ there exists a sequence $\{V_{0,j}\}_{1 \le j \le J}$ and $w_n^J$ in $L_{rad}^2$ and a family of parameters $(\ljn,\tjn)\in \mathbb{R}^+\times\mathbb{R}$ with
	$$
	\frac{\ljn}{\lam_{j',n}} + \frac{\lam_{j',n}}{\ljn} +\frac{|\tjn-t_{j',n}|}{\ljn^2} \xrightarrow{n\to\infty} \infty \quad  j\neq j'
	$$
	such that
	\begin{enumerate}
		\item[$(i)$]  $v_{0,n}=\sum_{j=1}^{J} \ljn^{-\frac32}V_j^l\left(-\frac{\tjn}{\ljn^2}, \frac{x}{\ljn}\right) + w_n^J, \quad V_j^l(t, x) = [e^{it\Delta}V_{0, j}](x)$,
		\item[$(ii)$] $\lim\limits_{J \to \infty}\underset{n\to\infty}{\limsup}\left\|e^{it\Delta}w_n^J\right\|_{L_t^qL_x^r(|x|^{-r\gam})} = 0$ for any $\gam$-admissible pair $(q, r)$ with $2 < q < \infty$,
		\item[$(iii)$]  $\|v_{0,n}\|_{L_x^2}^2 = \sum_{j=1}^{J}\|V_{0, j}\|_{L_x^2}^2 + \|w_n^J\|_{L_x^2}^2 + o(1)$ as $n \to \infty$,
	\end{enumerate}
\end{prop}
The next three subsections are devoted to showing Proposition \ref{profile} whose proof is divided into three steps: refined Strichartz estimate, prerequisite decompositions, proof of $L^2$-profile decomposition.
\subsection{Refined Strichartz estimate}
We first consider a refined Strichartz estimate.
\begin{prop}\label{prop-l2}
	Let $(q,r)$ is $\gam$-{\it admissible} and $q,r > 2$. Then there exist $\alpha,\,p$  such that $0 < \alpha < 1,\, 1 \le p < 2$, and
	\begin{align}\label{esti-paley}
	\|e^{it\Delta}f\|_{L_t^qL_x^r(|x|^{-r\gam})} \les \left( \sup_k 2^{3k\left(\frac12-\frac1{p}\right)}\|\widehat{P_kf}\|_{L_\xi^p}\right)^{\alpha}\|f\|_{L_x^2}^{1-\alpha}.
	\end{align}
\end{prop}
Here $P_k$ is the Littlewood-Paley projection on the annulus $\{|\xi| \sim 2^k\}$, $k \in \mathbb Z$. For the proof we need to improve the weighted Strichartz estimates for radial data.

\begin{lem}\label{refined}
 Let $0 < \gam < 1$ and $(q,r)$ satisfy that $2 \le q, r \le \infty$ and
 $$
 3\left( \frac12 - \frac1r\right) +\gam \le \frac2q \le 5(\frac12 - \frac1r) + \gam,\;\;(q,r) \neq \left(2,\frac{10}{3+2\gam}\right).
 $$
 Then for any  $f \in L_{rad}^2$ we have
	\begin{align}\label{refined-l2}
	\|e^{it\Delta}f\|_{L_t^qL_x^r(|x|^{-r\gamma})} \les \|f\|_{L_x^2}
	\end{align}
and
\begin{align}\label{li}
	\|e^{it\Delta}f\|_{L_t^qL_x^r(|x|^{-r\gam})} \les \left( \sum_{k \in \mathbb{Z}} \left(  2^{k\sigma}\|P_k f\|_{L_x^2}\right)^2\right)^{\frac12}
	\end{align}
	where $\sigma= 3\left( \frac12 - \frac1r\right) +\gam - \frac2q$.
\end{lem}
\begin{proof}[Proof of Lemma \ref{refined}]
By Strichartz estimates for the radial data \cite{chsl}, we have
$$
\|e^{it\Delta}f\|_{L_t^aL_x^b} \les \|f\|_{L_x^2}
$$
for $(a,b)$ such that $3 \left( \frac12 - \frac1b \right) \le \frac2a \le 5 \left( \frac12 - \frac1b \right)$ and $(a,b) \neq \left( 2,\frac{10}{3} \right)$. We also have $L^2$-weighted estimate of \cite{kato} such that
$$
\|e^{it\Delta}f\|_{L_{t,x}^2(|x|^{-2})} \les \|f\|_{L_x^2}.
$$
Then complex interpolation of weighted spaces (for instance see \cite{bergh}) yields
$$
\|e^{it\Delta}f\|_{L_t^qL_x^r(|x|^{-r\gam})} \les \|f\|_{L_x^2},
$$
where
\begin{align*}
\frac1q = \frac{1-\gam}a + \frac\gam2,\;\; \frac1r = \frac{1-\gam}b + \frac\gam2,\;\; 0 < \gam < 1.
\end{align*}
This implies \eqref{refined-l2}.

	By \eqref{refined-l2} we get
	$$
	\|e^{it\Delta}P_0f\|_{L_t^qL_x^r(|x|^{-r\gam})} \les \|f\|_{L_x^2}.
	$$
	Let $g = P_{k-1}f + P_k f + P_{k+1}f$. Then the above inequality yields that
	\begin{align*}
	\|e^{it\Delta}P_kf\|_{L_t^qL_x^r(|x|^{-r\gam})} &= \||x|^{-\gam}e^{it\Delta}P_k g\|_{L_t^qL_x^r}\\
	&= 2^{3k}\left\||x|^{-\gam}e^{i(2^{2k}t)\Delta}P_0\left[2^{-3k}g\left(\frac{\cdot}{2^k}\right)\right](2^k x)\right\|_{L_t^qL_x^r}\\
	&= 2^{3k  - \frac{3k}{r} +k\gam-\frac{2k}{q}}\left\|e^{it\Delta}P_0\left[2^{-3k}g\left(\frac{\cdot}{2^k}\right)\right]( x)\right\|_{L_t^qL_x^r(|x|^{-r\gam})}\\
	&\les 2^{3k  - \frac{3k}{r} +k\gam-\frac{2k}{q}}\left\|2^{-3k}g\left(\frac{\cdot}{2^k}\right)\right\|_{L_x^2}\\
	&\les 2^{k\left(  3(\frac12 - \frac1r) +\gam -\frac2q \right)}\left\|P_kf\right\|_{L_x^2}.
	\end{align*}
	This completes the proof of Lemma \ref{refined}.

\end{proof}

Now we prove the refined Strichartz estimate, Proposition \ref{prop-l2}.

\begin{proof}[Proof of Proposition \ref{prop-l2}]
	From Lemma \ref{stri-w} and Littlewood-Paley theory we get
	\begin{align}\label{esti-ltwo}
	\|e^{it\Delta}f\|_{L_t^qL_x^r(|x|^{-r\gam})} \les \|f\|_{L_x^2} \sim \left(\sum_k \|\widehat{P_kf}\|_{L_\xi^2}^2\right)^{\frac12}.
	\end{align}
	We will show the following estimates later.
	\begin{align}
	\|e^{it\Delta}f\|_{L_t^qL_x^r(|x|^{-r\gam})} &\les \left(\sum_k \left(2^{3k\left(\frac12 - \frac1{p_0}\right)}\|\widehat{P_kf}\|_{L_\xi^{p_0}}\right)^2\right)^{\frac12},\label{esti-space}\\
	\|e^{it\Delta}f\|_{L_t^qL_x^r(|x|^{-r\gam})} &\les \left(\sum_k \|\widehat{P_kf}\|_{L_\xi^2}^{q_0}\right)^{\frac1{q_0}},\label{esti-time}
	\end{align}
	for some $p_0,\,q_0$ with $p_0 < 2 < q_0$. Once these estimates hold, the interpolation of \eqref{esti-ltwo}, \eqref{esti-time}, and \eqref{esti-space} gives us
	$$
	\|e^{it\Delta}f\|_{L_t^qL_x^r(|x|^{-r\gam})} \les \left(\sum_k \left(2^{3k\left(\frac12 - \frac1{p_*}\right)}\|\widehat{P_kf}\|_{L_\xi^{p_*}}\right)^{q_*}\right)^{\frac1{q_*}}
	$$
	for any $(\frac1{q_*},\,\frac1{p_*})$ on triangle with vertices $(\frac12,\,\frac12),\,(\frac12,\frac1{p_0}),$ and $(\frac1{q_0},\frac12)$. This yields that
	\begin{align*}
	\|e^{it\Delta}f\|_{L_t^qL_x^r(|x|^{-r\gam})} &\le \left(\left(\sup_k2^{3k\left(\frac12 - \frac1{p_*}\right)}\|\widehat{P_kf}\|_{L_\xi^{p_*}}\right)^{q_*-2}\sum_k \left(2^{3k\left(\frac12 - \frac1{p_*}\right)}\|\widehat{P_kf}\|_{L_\xi^{p_*}}\right)^{2}\right)^{\frac1{q_*}}\\
	& \le \left(\sup_k2^{3k\left(\frac12 - \frac1{p_*}\right)}\|\widehat{P_kf}\|_{L_\xi^{p_*}}\right)^{\frac{q_*-2}{q_*}} \|f\|_{L_x^2}^{\frac2{q_*}}.
	\end{align*}
	We set $p = p_*$ and $\alpha = 1- \frac2{q_*}$. Then we get \eqref{esti-paley}.
	
	We now show \eqref{esti-space} and \eqref{esti-time}. By Lemma \ref{refined}, we get $\|e^{it\Delta}P_0f\|_{L_t^{\widetilde q}L_x^{\widetilde r}(|x|^{-{\widetilde r}\widetilde{\gam}})} \les \|P_0f\|_{L_x^2}$ for $(\widetilde q, \widetilde r, \widetilde{\gam})$ satisfying the condition of Lemma \ref{refined}.
By Hausdorff-Young inequality we also get $\|e^{it\Delta}P_0f\|_{L_{t,x}^{\infty}} \les \|\widehat{P_0f}\|_{L_\xi^1}$. Thus by complex interpolation we can find $0 < \theta = \theta(q, r, \gam) < 1$ and $(\widetilde q, \widetilde r, \widetilde{\gam})$ for each $(q, r, \gam)$ such that
$$
\frac1{q} = \frac{\theta}{\widetilde q},\;\; \frac1r = \frac\theta{\widetilde r},\;\; \gam = \theta\widetilde{\gam}.
$$
and $\|e^{it\Delta}P_0f\|_{L_t^{q}L_x^{ r}(|x|^{-{r}{\gam}})} \les \|\widehat{P_0f}\|_{L_\xi^{p_0}}$, where $\frac1{p_0} = \frac{1-\theta}2$.
Using \eqref{li}, we get
	$$
	\|e^{it\Delta}P_kf\|_{L_t^{q}L_x^{r}(|x|^{-r\gam})} \les  2^{\frac32k} \left\|e^{it\Delta}\left[P_0 f\left(\frac{\cdot}{2^k}\right)\right]\right\|_{L_t^{q}L_x^{r}(|x|^{-r\gam})} \les  2^{3k\left(\frac12 - \frac1{p_0}\right)}\|\widehat{P_0f}\|_{L_\xi^{p_0}}.
	$$
	 Therefore we have
	$$
	\|e^{it\Delta}f\|_{L_t^{q}L_x^{r}(|x|^{-r\gam})} \le \left( \sum_{k} \|e^{it\Delta}P_k f\|_{L_t^{q}L_x^{r}(|x|^{-r\gam})}^2 \right)^{\frac12}  \les \left( \sum_{k} \left( 2^{3k\left( \frac12 - \frac1{p_0}\right)}\|\widehat{P_k f}\|_{L_\xi^{p_0}}\right)^2 \right)^{\frac12}.
	$$
	
	We now turn to the \eqref{esti-time}. By (2.10) of \cite{chhwkwle}, we get
	\begin{align}\label{lfour}
	\|e^{it\Delta}f\|_{L_{t,x}^4} \les  \left( \sum_{k} \left( 2^{\frac{k}4}\|\widehat{P_k f}\|_{L_\xi^2} \right)^4 \right)^{\frac14}.
	\end{align}
	By the same argument as above for each $(q, r, \gam)$ we can find $\theta$ and $(\widetilde q, \widetilde r, \widetilde{\gam})$ such that
	$$
	\frac1q = \frac\theta4 + \frac2{\widetilde q}(1-\theta) ,\;\;   	\frac1r = \frac\theta4 + \frac2{\widetilde r}(1-\theta).
	$$
	Hence complex interpolation of \eqref{refined-l2} and \eqref{lfour} yields \eqref{esti-time}.
		This completes the proof of Proposition \ref{prop-l2}.
\end{proof}


\subsection{Prerequisite decompositions}
At first we consider a decomposition associated with refined Strichartz estimate.
\begin{lem}\label{decom-space}
	Let $\{u_n\}$ be a sequence of complex-valued function and $\|u_n\|_{L_x^2} \le A$. Then for any $\eta >0$, there exists $J=J(\eta)$, $\kappa_{j,n} \in (0,\infty)$ and $\{f_{j,n}\} \subset L_x^2$ such that
	$$
	u_n = \sum_{j=1}^J f_{j,n} + q_n^J
	$$
	with
	\begin{enumerate}
		\item[$(i)$]   there exists compact set $K=K(J) \subset \{ \xi : r_1 <|\xi|<r_2 \}$ satisfying that
		$$
		(\kjn)^\frac32|\widehat{f_{j,n}}(\kjn \xi)|\le C_\eta \chi_{K}(\xi) \quad \mbox{for every} \quad 1 \le j \le J,
		$$
		\item[$(ii)$]  {\rm orthogonality:} $\frac{\kjn}{\kappa_{j',n}} + \frac{\kappa_{j',n}}{\kjn}  \xrightarrow{n\to\infty} 0 \quad  j\neq j'$,
		\item[$(iii)$] $\underset{n\to\infty}{\limsup}\left\| e^{it\Delta}q_n^J\right\|_{L_t^qL_x^r(|x|^{-r\gam})} \le \eta $ for any $J\ge1$,
		\item[$(iv)$]  $\|u_n\|_{L_x^2}^2 = \sum_{j=1}^{J}\|f_{j,n}\|_{L_x^2}^2 + \|q_n^J\|_{L_x^2}^2 + o(1)$ as $n \to \infty$.
	\end{enumerate}
\end{lem}

\begin{proof}
If $\left\| e^{it\Delta}u_n\right\|_{L_t^qL_x^r(|x|^{-r\gam})} \le \eta$, the proof has been done. Thus we can assume that
$$
\left\| e^{it\Delta}u_n\right\|_{L_t^qL_x^r(|x|^{-r\gam})} > \eta.
$$
By Proposition \ref{prop-l2} there exists an annulus $A_{1,n} = \{ \xi : \frac{\kon}2 < |\xi| < \kon\}$ such that
$$
C_1(\kon)^{3(\frac1{p} -\frac12)} \eta^{\frac1\alpha} \le \|\widehat{u_{1,n}}\|_{L_x^{p}}
$$
where $\widehat{u_{1,n}} = \widehat{u_n}\chi_{A_{1,n}}$. And
$$
\int_{|\widehat{u_{1,n}}| > \rho} |\widehat{u_{1,n}}|^{p}\, d\xi = \int_{|\widehat{u_{1,n}}| > \rho} (\rho^{2-{p}}|\widehat{u_{1,n}}|^{p})\rho^{p-2}\, d\xi \le A\rho^{p-2}
$$
for any $\rho>0$ and hence
$$
\left(\int_{|\widehat{u_{1,n}}| > \rho} |\widehat{u_{1,n}}|^{p}\, d\xi\right)^{\frac1{p}} \le A^\frac1p \rho^{1-\frac2{p}}.
$$
We set $\rho =\left( A^{-\frac1{p}}\frac{C_1}2 (\kon)^{3(\frac1{p}-\frac12)} \eta^{\frac1\alpha}\right)^{\frac{p}{2-p}}$. Then we have
$$
\frac{C_1}2 (\kon)^{3(\frac1{p}-\frac12)} \eta^{\frac1\alpha} \le \left(\int_{|\widehat{u_{1,n}}| \le \rho} |\widehat{u_{1,n}}|^{p}\, d\xi\right)^{\frac1{p}} \le C(\kon)^{3(\frac1{p} -\frac12)}\left(\int_{|\widehat{u_{1,n}}| \le \rho} |\widehat{u_{1,n}}|^2\, d\xi\right)^{\frac12}.
$$
This implies that
$$
\frac{C_1'}2\eta^{\frac1\alpha} \le  \left(\int_{|\widehat{u_{1,n}}| \le \rho} |\widehat{u_{1,n}}|^2\, d\xi\right)^{\frac12}.
$$

Let us define $G_n^1(\psi) = (\kon)^{\frac32}\psi(\kon\cdot)$. Then
$$
\|v_{1,n}\|_{L_x^2} \ge \frac{C_1'}2 \eta^{\frac1\alpha} \;\; \mbox{and} \;\; |G_n^1( \widehat{ v_{1,n}}(\xi) )| = (\kon)^{\frac32}\widehat{ v_{1,n}}(\kon\xi) \le C_\eta \chi_{A}(\xi),
$$
where $\widehat{ v_{1,n}} =  \widehat{u_{1,n}} \chi_{\{\widehat{u_{1,n}} < \rho\}}$ and $ \mathcal A = \{\xi : \frac12 <|\xi|< 1 \}$. If $\left\| u_n -v_{1,n}\right\|_{L_t^qL_x^r(|x|^{-r\gam})} \le \eta$, we stop here, and if $\left\| u_n -v_{1,n}\right\|_{L_t^qL_x^r(|x|^{-r\gam})} > \eta$, we repeat the above process with $u_n -v_{1,n}$. After repeating $J$ times, we get
\begin{align*}
&u_n = \sum_{j=1}^Jv_{j,n} + q_n^J,\\
&\|u_n\|_{L_x^2}^2 = \sum_{j=1}^J\|v_{j,n}\|_{L_x^2}^2  + \|q_n^J\|_{L_x^2}^2,\\
&\|e^{it\Delta}q_n^J\|_{L_t^qL_x^r(|x|^{-r\gam})} \le \eta.
\end{align*}

Now let us consider the property $(iii)$. We say that $\kjn,\kappa_{j',n}$ are orthogonal if and only if $\limsup ( \frac{\kjn}{\kappa_{j',n}} + \frac{\kappa_{j',n}}{\kjn})= \infty$. Let us define $f_{1,n}$ to be the sum of $v_{j,n}$ satisfied $\kjn$ are not orthogonal to $\kon$. Set least $j_0 \in [2,J]$ such that $\kappa_{j_0,n},\,\kon$ are orthogonal. And let us define $f_{2,n}$ to be sum of $v_{j,n}$ satisfied $\kjn$ are not orthogonal to $\kappa_{j_0,n}$ but orthogonal to $\kon$. After repeating finite step, we have $\{f_{j,n}\}$ satisfying that
\begin{align*}
&u_n = \sum_{j=1}^Jf_{j,n} + q_n^J,\\
&\|u_n\|_{L_x^2}^2 = \sum_{j=1}^J\|f_{j,n}\|_{L_x^2}^2  + \|q_n^J\|_{L_x^2}^2\\
&\limsup_{n\to\infty} \left( \frac{\kjn}{\kappa_{j',n}} + \frac{\kappa_{j',n}}{\kjn}\right)= \infty.
\end{align*}

Finally, we need to show $(i)$. Since $v_{j,n}$ collected in $f_{1,n}$ has $\kjn$ which is not orthogonal to $\kon$, we obtain $\limsup( \frac{\kjn}{\kappa_{j',n}} + \frac{\kappa_{j',n}}{\kjn}) < \infty$. We also know that $|G_n^j(\widehat{ v_{1,n}})| \le C_\eta \chi_{\mathcal A}$. Hence by scaling and non-orthogonality, we get $|G_n^j(\widehat{ v_{1,n}})| \le \widetilde{C_\eta}\chi_{\widetilde{A}}$ where $\widetilde{\mathcal A} = \{ \xi : r_1 <|\xi|< r_2\}$ for some $r_1,\,r_2 >0$. The proof has been finished.
\end{proof}

We next complement a further decomposition w.r.t. time parameter.
\begin{lem}\label{decom-time}
Let $\{f_n\} \subset L_x^2$ satisfy $(\mu_n)^{\frac32}|\widehat{f_n}(\kappa_n \xi)| \le \widehat{F}(\xi)$ for some $\widehat{F} \in L_\xi^{\infty}(K)$ with compact set $K\subset \mathcal A = \{\xi : 0 <r_1 <|\xi|<r_2\}$. Then there exist $\{\tau_{j,n}\} \subset \mathbb{R}$ and $\{V^l\} \subset L_x^2$ such that
\begin{enumerate}
		\item[$(i)$]  {\rm orthogonality:} $\limsup_{n\to\infty} |\tau_{j,n} - \tau_{j',n}| = \infty \qquad j\neq j'$,
		\item[$(ii)$]  for every $M>0$, there exists $e_n^M \in L_x^2$ such that
$$
f_n(x) = \sum_{l=1}^M (\kappa_n)^{\frac32}\left(e^{i\tau_{j,n}\Delta}V^l\right)(\kappa_n x) + e_n^M(x) \;\; \mbox{and}\;\; \limsup_{M\to\infty, n\to\infty}\|e^{it\Delta}e_n^M\|_{L_t^qL_x^r(|x|^{-r\gam})} =0,
$$
		\item[$(iii)$]  $\|f_n\|_{L_x^2}^2 = \sum_{l=1}^{M}\|V^l\|_{L_x^2}^2 + \|e_n^M\|_{\h}^2 + o(1)$ as $n \to \infty$.
	\end{enumerate}
\end{lem}

\begin{proof}
Let us denote by $\mathcal{C}$ the collection of functions $\{F_n\}$ which are given by $\widehat{F_n}(\xi) = (\kappa_n)^{\frac32}\widehat{f_n}(\kappa_n \xi)$, and define
$$
\mathcal{L}(\mathcal{C})=\{\mbox{weak-lim } e^{-i\tau_{1,n}\Delta}F_n\;\; \mbox{in} \;\; L_x^2 \;:\; \tau_{1,n} \in \mathbb{R} \}
$$
and $\mu(\mathcal{C}) = \sup_{V\in \mathcal{L}(\mathcal{C})}\|V\|_{L_x^2}$. Then $\mu(\mathcal{C}) \le \limsup\|F_n\|_{L_x^2}$.


Let us choose a subsequence $\{F_n\}$, $\tau_{1,n}$ and $V^1$ such that $e^{-i\tau_{1,n}\Delta}F_n \rightharpoonup V^1 $ as $n \to \infty$ and $\|V^1\|_{L_x^2} \ge \frac12 \mu(\mathcal{C})$. Let $F_n^1 = F_n - e^{i\tau_{1,n}\Delta}V^1$ and $\mathcal{C}^1 = \{F_n^1\}$. Then
\begin{align*}
\limsup_{n\to\infty} \|F_n^1\|_{L_x^2}^2 &= \limsup_{n\to\infty} \langle F_n - e^{i\tau_{1,n}\Delta}V^1,F_n - e^{i\tau_{1,n}\Delta}V^1 \rangle\\
&= \limsup_{n\to\infty} \langle e^{-i\tau_{1,n}\Delta}F_n - V^1 , e^{-i\tau_{1,n}\Delta}F_n - V^1 \rangle\\
&= \limsup_{n\to\infty} \left( \langle F_n,F_n \rangle + \langle V^1 - e^{-i\tau_{1,n}\Delta}F_n, V^1 \rangle + \langle V^1 , V^1 - e^{-i\tau_{1,n}\Delta}F_n \rangle - \langle V^1,V^1 \rangle  \right)\\
&= \limsup_{n\to\infty} \left( \|F_n\|_{L_x^2}^2 - \|V^1\|_{L_x^2}^2\right)
\end{align*}
where $\langle \cdot,\cdot \rangle$ is $L_x^2$-inner product. We repeat the process replacing $F_n^1$ with $F_n^2$. By taking a diagonal sequence we may write
$$
F_n(x) = \sum_{l=1}^M  e^{i\tau_{1}^n\Delta}V^l + F_n^M,
$$
which satisfies that
$$
\limsup_{n\to\infty} \|F_n\|_{L_x^2}^2 = \sum_{l=1}^M\|V^l\|_{L_x^2}^2 + \limsup_{n\to\infty} \|F_n^M\|_{L_x^2}^2.
$$
Then $\sum_{l=1}^M\|V^l\|_{L_x^2}^2$ is convergent. Thus it yields $\limsup_{n\to\infty}\|V^l\|_{L_x^2} =0 $. Since $\mu(\mathcal{C}^M) \le 2 \|V^{M+1}\|_{L_x^2} $, we get $\limsup_{M\to\infty}\mu(\mathcal{C}^M) =0 $.

 Next we define $e_n^M$ by $\widehat{e_n^M} = \kappa_n^{-\frac32}\widehat{F_n^M}$. Then we have only to show that
 \begin{align}\label{erterm}
 \limsup_{n\to\infty}\|e^{it\Delta}e_n^M\|_{L_t^qL_x^r(|x|^{-r\gam})} \les \mu(\mathcal{C}^M)^\theta
 \end{align}
for some $\theta \in (0,1)$. By construction we may assume that $\widehat{V^l}$ have compact support $K$. Since the pair $(q,r)$ is $\gam$-{\it admissible}, we get
$$
\|e^{it\Delta}e_n^M\|_{L_t^qL_x^r(|x|^{-r\gam})} = \|e^{it\Delta}F_n^M\|_{L_t^qL_x^r(|x|^{-r\gam})} \le \|e^{it\Delta}F_n^M\|_{L_t^{\theta q} L_x^{\theta r}(|x|^{-r\gam})}^{\theta} \|e^{it\Delta}F_n^M\|_{L_{t,x}^\infty}^{1-\theta}
$$
for any $\theta \in (0,1)$. We choose $\theta$ such that $0 < 1- \theta \ll 1$ and $3\left( \frac12 - \frac1{\theta r}\right) + \frac{\gam}{\theta}< \frac2{\theta q} \le 5\left( \frac12 - \frac1{\theta r}\right) + \frac{\gam}{\theta} $. Hence by \eqref{li} we obtain
\begin{align*}
\|e^{it\Delta}F_n^M\|_{L_t^{\theta q} L_x^{\theta r}(|x|^{-r\gam})} &\les  r_1^{3(\frac12 - \frac1{\theta r}) + \frac{\gam}{\theta}}\|F_n^M\|_{L_x^2} \les  r_1^{3(\frac12 - \frac1{\theta r}) + \frac{\gam}{\theta}}.
\end{align*}
Now it suffices to prove $\limsup_{n\to\infty}\|e^{it\Delta}F_n^M\|_{L_{t,x}^\infty} \les \mu(\mathcal{C}^M)$. For this we assume that
$$
\limsup_{\substack{M\to\infty\\ n\to\infty} }\|e^{it\Delta}F_n^M\|_{L_{t,x}^\infty} > \delta
$$
for some $\delta>0$. Let us choose $(\tau_{j}^M, y_n^M)$ such that $\|e^{it\Delta}F_n^M\|_{L_{t,x}^\infty} = |e^{i\tau_{j}^M\Delta}F_n^M(y_n^M)|$. Then we can show that $|y_n^M|$ is uniformly bounded. In fact,
\begin{align*}
|e^{it\Delta}F_n^M(x_1) - e^{it\Delta}F_n^M(x_2)| &\le \left( \sup |\nabla(e^{it\Delta}F_n^M(y))|\right)|x_1 - x_2|\\
& \le \int |\xi||e^{it|\xi|^2} \widehat{F_n^M}(\xi)| \,d\xi |x_1-x_2|\\
&\les \left(\int_0^{r_2}r^{n+1}dr\right)^{\frac12} \|F_n^M\|_{L_x^2}|x_1- x_2|\\
&\les r_2^{\frac{n+2}2 |x_1- x_2|}.
\end{align*}
This yields that $|e^{i\tau_{n}^M\Delta}F_n^M(y)| > \frac{\delta}2$ if $|y-y_n^M| \le C\frac{\delta}2$ for some $C$. Since $e^{i\tau_{n}^M\Delta}F_n^M(y)$ is radial, we get
$$
|e^{i\tau_{n}^M\Delta}F_n^M(y)| > \frac{\delta}2
$$
for $|y_n^M|-C\frac{\delta}2 < |y| < |y_n^M|+C\frac{\delta}2$. Then we have $\frac{\delta}2|y_n^M|^{n-1}\frac{\delta}2C \le \|F_n^M\|_{L_x^2} \le 1$, which implies that $|y_n^M|$ is uniformly bounded. Since $|y_n^M|$ is uniformly bounded, there exists $y_0^M$ such that $ y_n^M \to y_0^M$ as $n \to \infty$. Then $|e^{i\tau_{j}^M\Delta}F_n^M(y_0^M)|\ge\frac12|e^{i\tau_{j}^M\Delta}F_n^M(y_n^M)|$. For $\psi \in C_0^\infty(\mathbb{R}^3)$, we denote $\psi^M$ by $\widehat{\psi^M} = \psi \widehat{\delta_{y_0^M}}$ where $\delta_{y_0^M}$ is Dirac-delta measure. Hence we obtain
\begin{align*}
\limsup_{n\to\infty} \|e^{it\Delta}F_n^M\|_{L_{t,x}^\infty} &\les \limsup_{n\to\infty} |e^{i\tau_{j}^M\Delta}F_n^M(y_0^M)|\\
&= \limsup_{n\to\infty} \left| \int e^{i\tau_{j}^M\Delta}F_n^M(y)\psi^M(y)  \,dy \right|\\
&\le \|\psi^M\|_{L_x^2} \mu(\mathcal{C}^M)\\
&\le \mu(\mathcal{C}^M).
\end{align*}
This ends the proof.
\end{proof}

\subsection{Proof of $L^2$-profile decomposition}
We now prove the Proposition \ref{profile}.

From Lemma \ref{decom-space} and \ref{decom-time}, we get
\begin{align}\label{decomposition}
v_{0,n} = \sum_{j=1}^J\sum_{l=1}^{M_j} V_n^{j,l} + \omega_n^{J,M_1,\cdots,M_J}
\end{align}
where
\begin{align*}
V_n^{j,l} = e^{it_n^{j,l}\Delta}\left[\ljn^{-\frac32}V_0^{j,l}\left(\frac{\cdot}{\ljn}\right)\right], \qquad \omega_n^{J,M_1,\cdots,M_J} = \sum_{j=1}^Je_n^{j,M_j} + q_n^J
\end{align*}
with $\ljn=\kjn^{-1}$ and $t_n^{j,l} = \kjn^{-2} s_n^{j,l}$. Then by construction in Lemmas \ref{decom-space}, \ref{decom-time} the followings are held.
\begin{align*}
&(i)\;\; \mbox{The pairs}\;\; (\ljn, t_n^{j,l}) \;\;\mbox{are pairwise orthogonal},\\
&(ii)\;\; \|v_{0,n}\|_{L_x^2}^2 = \sum_{j=1}^J\sum_{l=1}^{M_j} \|V_n^{j,l}\|_{L_x^2}^2 + \|\omega_n^{J,M_1,\cdots,M_J}\|_{L_x^2}^2 + o(1) \;\;\mbox{as}\;\; n \to \infty\\
& \qquad \quad \mbox{with}\;\;\|\omega_n^{J,M_1,\cdots,M_J}\|_{L_x^2}^2 = \sum_{j=1}^J\|e_n^{j,M_j}\|_{L_x^2}^2 + \|q_n^J\|_{L_x^2}^2.
\end{align*}

Now it remains to prove that
$$
\underset{n\to\infty}{\limsup}\left\| e^{it\Delta}w_n^{J,M_1,\cdots,M_J}\right\|_{L_t^qL_x^r(|x|^{-r\gam})} \to 0.
$$
Let $\eta > 0$ be any given number. From Lemma \ref{decom-space} we get
$$
\underset{n\to\infty}{\limsup}\left\| e^{it\Delta}q_n^J\right\|_{L_t^qL_x^r(|x|^{-r\gam})} \le \eta
$$
for $J \ge N_1$ with $N_1$ enough large. And from Lemma \ref{decom-time}, we have
$$
\underset{n\to\infty}{\limsup}\left\|e^{it\Delta} e_n^{j,M_j}\right\|_{L_t^qL_x^r(|x|^{-r\gam})} \le \eta
$$
for $M_j \ge N_2$ with $N_2$ enough large.
We can rewrite the remainder as
$$
\omega_n^{J,M_1,\cdots,M_J} = \sum_{1\le j \le J}e_n^{j,\max\{M_j,N_2\}} + R_n^{J,M_1,\cdots,M_J} + q_n^J,
$$
where
$$
R_n^{J,M_1,\cdots,M_J} = \sum_{\substack{1\le j \le J\\ M_j < N_2}} (e_n^{j,M_j} - e_n^{j,N_2}) = \sum_{\substack{1\le j \le J\\ M_j < N_2}}\sum_{l=M_j}^{N_2}V_n^{l,j}.
$$
To handle this we recall an orthogonality w.r.t. space-time norm.
\begin{lem}\label{limsup}
	Assume that $e^{it\Delta}f_n^{l,j} \in L_t^qL_x^r(|x|^{-r\gam})$ for all $n, l, j \ge 1$. Then for any $J, m \ge 1$
	\begin{align*}
	\limsup_{n\to\infty} \left\|\sum_{j=1}^J \sum_{l=1}^{m} e^{it\Delta}f_n^{l,j} \right\|_{L_t^qL_x^r(|x|^{-r\gam})}^2 \le \sum_{j=1}^J \sum_{l=1}^{m}\limsup_{n\to\infty}\|e^{it\Delta}f_n^{l,j}\|_{L_t^qL_x^r(|x|^{-r\gam})}^2.
	\end{align*}	
\end{lem}
By Lemma \ref{limsup}, we obtain
\begin{align*}
\underset{n\to\infty}{\limsup} \|e^{it\Delta}  \omega_n^{J,M_1,\cdots,M_J}\|_{L_t^qL_x^r(|x|^{-r\gam})} &\le \underset{n\to\infty}{\limsup} \sum_{j=1}^J\|e^{it\Delta} e_n^{j,\max\{M_j,N_2\}}\|_{L_t^qL_x^r(|x|^{-r\gam})}\\
&\qquad + \underset{n\to\infty}{\limsup}\|e^{it\Delta} R_n^{J,M_1,\cdots,M_J}\|_{L_t^qL_x^r(|x|^{-r\gam})}\\
&\qquad + \underset{n\to\infty}{\limsup}\|e^{it\Delta} q_n^J\|_{L_t^qL_x^r(|x|^{-r\gam})}\\
&\le (J+1)\eta + \underset{n\to\infty}{\limsup}\|e^{it\Delta} R_n^{J,M_1,\cdots,M_J}\|_{L_t^qL_x^r(|x|^{-r\gam})}\\
&\le (J+1)\eta +  \left(\sum_{\substack{1\le j \le J\\ M_j < N_2}}\sum_{l=M_j}^{N_2}\underset{n\to\infty}{\limsup} \|e^{it\Delta}V_n^{l,j}\|_{L_t^qL_x^r(|x|^{-r\gam})}^2 \right)^{\frac12}\\
&\le (J+1)\eta +  \left(\sum_{1\le j \le J}\sum_{l=M_j}^{\infty} \|V_n^{l,j}\|_{L_x^2}^2 \right)^{\frac12}\\
&\le (J+2)\eta
\end{align*}
for $J,\,M_j$ enough large. This completes the proof of Proposition \ref{profile}.

\begin{proof}[Proof of Lemma \ref{limsup}]
	It is suffice to prove that for $(l,j) \neq (l',j')$,
	$$
	\limsup_{n\to\infty} \left\| e^{it\Delta}f_n^{l,j} e^{it\Delta}f_n^{l',j'}\right\|_{L_t^{\frac{q}2}L_x^{\frac{r}2}(|x|^{-r\gam})} = 0.
	$$
	Let $(\ljn,t_n^{l,j})$ satisfy that
	\begin{align*}
	\limsup_{n\to\infty} \left( \frac{\lam_{j',n}}{\ljn} + \frac{\ljn}{\lam_{j',n}}\right) = \infty \;\;\mbox{and}\;\;  \limsup_{n\to\infty} \frac{|t_n^{l,j} - t_n^{l',j'}|}{\ljn^2} =\infty.
	\end{align*}
	Now we will prove that if $\Psi_1,\, \Psi_2 \in C_0^{\infty}$, then
	$$
	\limsup_{n\to\infty} \left\| \ljn^{-\frac32} \Psi_1 \left( \frac{\cdot- t_n^{l,j}}{\ljn^2} , \frac{\cdot}{\ljn} \right) \lam_{j',n}^{-\frac32} \Psi_2 \left( \frac{\cdot- t_n^{l',j'}}{\lam_{j',n}^2} , \frac{\cdot}{\lam_{j',n}} \right)\right\|_{L_t^{\frac{q}2}L_x^{\frac{r}2}(|x|^{-r\gam})} = 0.
	$$
	Since $(q,r)$ is $\gam$-{\it admissible}, we get
	\begin{align*}
	A_n &:=\left\| \ljn^{-\frac32} \Psi_1 \left( \frac{\cdot- t_n^{l,j}}{\ljn^2} , \frac{\cdot}{\ljn} \right) \lam_{j',n}^{-\frac32} \Psi_2 \left( \frac{\cdot- t_n^{l',j'}}{\lam_{j',n}^2} , \frac{\cdot}{\lam_{j',n}} \right)\right\|_{L_t^{\frac{q}2}L_x^{\frac{r}2}(|x|^{-r\gam})}\\
	&\le \left\| \ljn^{-\frac32} \left\| \Psi_1 \left( \frac{\cdot- t_n^{l,j}}{\ljn^2} , \frac{\cdot}{\ljn} \right)\right\|_{L_x^r(|x|^{-r\gam})}  \lam_{j',n}^{-\frac32} \left\| \Psi_2 \left( \frac{\cdot- t_n^{l',j'}}{\lam_{j',n}^2} , \frac{\cdot}{\lam_{j',n}} \right)\right\|_{L_x^r(|x|^{-r\gam})}\right\|_{L_t^{\frac{q}2}}\\
	&= \left\| \ljn^{-\frac32 +\frac3r -\gam} \left\| \Psi_1 \left( \frac{\cdot- t_n^{l,j}}{\ljn^2} , \cdot \right)\right\|_{L_x^r(|x|^{-r\gam})}  \lam_{j',n}^{-\frac32+\frac3r -\gam} \left\| \Psi_2 \left( \frac{\cdot- t_n^{l',j'}}{\lam_{j',n}^2} , \cdot \right)\right\|_{L_x^r(|x|^{-r\gam})}\right\|_{L_t^{\frac{q}2}}\\
	&= \left\| \ljn^{-\frac32 +\frac3r -\gam +\frac4q} \left\| \Psi_1 \left( \cdot , \cdot \right)\right\|_{L_x^r(|x|^{-r\gam})}  \lam_{j',n}^{-\frac32+\frac3r -\gam} \left\| \Psi_2 \left( \left(\frac{\ljn}{\lam_{j',n}}\right)^2\cdot - \frac{t_n^{l',j'}-t_n^{l,j}}{\lam_{j',n}^2} , \cdot \right)\right\|_{L_x^r(|x|^{-r\gam})}\right\|_{L_t^{\frac{q}2}}\\
	&\le \left\| \left(\frac{\ljn}{\lam_{j',n}} \right)^{\frac2q} \left\| \Psi_1 \left( \cdot , \cdot \right)\right\|_{L_x^r(|x|^{-r\gam})}   \left\| \Psi_2 \left( \left(\frac{\ljn}{\lam_{j',n}}\right)^2\cdot - \frac{t_n^{l',j'}-t_n^{l,j}}{\lam_{j',n}^2} , \cdot \right)\right\|_{L_x^r(|x|^{-r\gam})}\right\|_{L_t^{\frac{q}2}}.
	\end{align*}
	Since the time support of $\|\Psi_2(t,\cdot)\|_{L^r(|x|^{-r\gam})}$ is compact, $\limsup_{n\to\infty} A_n =0$. By density we get the desired result.
\end{proof}


\section{Proof of the main theorems}

We are now ready to show main theorems. We follow the standard approach developed in \cite{km}: variational estimate;  existence and compactness of minimal energy blowup solution; rigidity. However, the variational estimates do not depend on the range of $b$ and our proof of two remaining parts is very similar to that of \cite{chkl3} except for the weighted space norms. By replacing the norms $S(I), \, W_i(I)$ appearing in \cite{chkl3} with weighted space norms $\s, \, \yi$ together with Proposition \ref{per} and Theorem \ref{profile-thm} one can follow up the full proof introduced in Section 5 of \cite{chkl3} without difficulty. Hence we leave the details to the readers and here we only sketch them without proof.

\subsection{Variational estimates}\label{sec-vari}

\begin{lem}[Energy trapping]\label{trapping}
	Let $u$ be a solution of \eqref{maineq} with $\varphi$  such that
	$$
	g_s\|\varphi\|_{\h}^2 < \|Q_b\|_{\h}^2\;\;\mbox{and}\;\; \eg(\varphi)\le(1-\delta_0)\eg(Q_b)
	$$
	for some $\delta_0 >0$. Then there exits $\bar{\delta}=\bar{\delta}(\delta_0)$ such that
	\begin{align*}
	&(i)   \;g_s\|u(t)\|_{\h}^2 \le (1-\bar{\delta})\|Q_b\|_{\h}^2,\\
	&(ii)  \;\int |\nabla u(t)|^2 - g|u(t)|^{p+1} dx \ge \bar{\delta}\int |\nabla u(t)|^2 dx,\\
	&(iii) \;{\rm (Coercivity)}\;\eg(u(t)) \sim \|u(t)\|_{\h}^2 \sim \|\varphi\|_{\h}^2,
	\end{align*}
	for all $t \in I^*$, where $I^*$ is the maximal existence time interval.
\end{lem}
\begin{lem}\label{blowup-cor}
	Let $u$ be a solution of \eqref{maineq} with $\varphi$  such that
	$$
	\gs \|\varphi\|_{\h}^2 \ge \|Q_b\|_{\h}^2\;\;\mbox{and}\;\; E_g(\varphi) \le (1-\delta_0) E_g(Q_b)
	$$
	for some $\delta_0 > 0$. Then there exits $\bar{\delta}=\bar{\delta}(\delta_0)$ such that
	\begin{align*}
	&(i)   \;\gs\|u(t)\|_{\h}^2 \ge (1+\bar{\delta})\|Q_b\|_{\h}^2,\\
	&(ii)  \;\int |\nabla u(t)|^2 -(1-\eta) g|u(t)|^{p+1} dx \le -\frac{(2-b-(3-b)\eta)\bar{\delta}}{\gs}\|Q_b\|_{\h}^2
	\end{align*}
	for all $t \in I^*$ and $0 \le \eta \le k_g$, where $I^*$ is the maximal existence time interval and  $k_g = \frac{2-b - g_0}{3 - b - g_0}$, and $g_0 = \gs(3-b - g_i)$.
\end{lem}

\subsection{Minimal energy blowup solution}


For each $0 < e < \eg(Q_b)$ let
$$
\mathcal A(e) := \left\{\varphi \in \hr : E_g(\varphi) < e, \gs\|\varphi\|_{\h}^2 < \|Q_b\|_{\h}^2 \right\}$$
and let
$$
\mathcal \beta(e) := \sup\Big\{\|v\|_{\sst} : v(0) \in \mathcal A(e), v \;\;\mbox{solution to}\;\;\eqref{maineq} \Big\}.
$$
Define $E_{g,c} = \sup\{e : \mathcal \beta(e) < +\infty\}$. In view of the blowup criterion and small data scattering (Remarks \ref{blowup-cri} and \ref{scattering}) we deduce that $0 < E_{g,c} \le E_g(Q_b)$. We assume that $E_{g,c} < \eg(Q_b)$, which will lead us to a contradiction.

At this point, we may expect that $\eg(Q_b)$ is critical value between GWP and blowup.
\begin{prop}[Existence of minimal energy blowup solution]\label{mebs}
	Let $\varphi_c \in \hr$ satisfy that $E_g(\varphi_c) = \egc(<E_g(Q_b))$ and $\gs\|\varphi_c\|_{\h}^2 < \|Q_b\|_{\h}^2$. If $u_c$ is the corresponding solution to \eqref{maineq}, then $\|u_c\|_{\sst} = +\infty$.
\end{prop}
The solution $u_c$ is called the minimal energy blowup solution (MEBS).

\begin{prop}[Compactness of the MEBS flows]\label{compactness}
	For any $u_c$ as in Proposition \ref{mebs}, with $\|u_c\|_{\sst} = +\infty$, there exist $\lambda(t) \in \mathbb{R}^+$, $t \in I_+^* (I_+^* := I^* \cap \left[0,\infty \right))$ such that
	$$
	\mathcal M = \left\{ v(x,t) := \lambda(t)^{-\frac12} u_c\left(t, \frac{x}{\lambda(t)}\right) : t \in I_+^*  \right\}
	$$
	has compact closure in $\hr$.
\end{prop}

\subsection{Rigidity}

\begin{prop}\label{rigidity}
	Suppose that $g$ is nonnegative, bounded radial function satisfying the conditions \eqref{scaling}, \eqref{var-con}, and \eqref{rig-con}. Let $\vp \in \hr$ satisfy that $\eg(\vp) < \eg(Q_b)$ and $\gs\|\vp\|_{\h}^2 < \|Q_b\|_{\h}^2$.  Let $u$ be the corresponding solution to \eqref{maineq} with $\vp$ and let $I^* = (-T_-,T_+)$ be the maximal existence time interval. Assume there exists $\lam(t) > 0$ such that
	$$
	\mathcal M:= \left\{ v(t, x) = (\lam(t))^{-\frac12}u\left(t,\frac{x}{\lam(t)}\right) : t \in [0,T_+)\right\}
	$$
	has compact closure in $\hr$. 
	Then $T_+ = +\infty$ and $\vp = 0$.
\end{prop}

\subsection{Proof of Theorem \ref{mainresult}}
By the definition of $E_{g,c}$ we deduce that
\begin{enumerate}
	\item If $0 \le e <E_{g,c}$, $\gs\|\varphi\|_{\h}^2 < \|Q_b\|_{\h}^2$, and $\eg(\varphi) < e$, then $\|u\|_{\sst} < +\infty$.
	
	\item If $E_{g,c} \le e < \eg(Q_b)$, $\gs\|\varphi\|_{\h}^2 < \|Q_b\|_{\h}^2$, and  $E_{g,c} \le \eg(\varphi)< e <\eg(Q_b)$, then $\|u\|_{\sst} = +\infty$.
\end{enumerate}
On the other hand, we can remove MEBS $u_c$ by Proposition \ref{rigidity} under the condition $\egc < \eg(Q_b)$. This implies that $\egc = \eg(Q_b)$. Therefore by (1) above we concluded that \eqref{maineq} is globally well-posed under the assumption of Theorem \ref{mainresult}.

\subsection{Proof of Theorem \ref{blowup-thm}}
We first show the part $(1)$. Let $\psi(x)\in C_0^\infty(\rt)$  and $\psi_r(x)$ be as follows:
\begin{align*}
\psi(x)&:=\left\{ \begin{array}{cc}
|x|^2 & (|x|\leq 1) \\
0 & (|x|\geq 10)
\end{array}  \right.,\qquad
\psi_r(x):=r^2\psi(\frac{x}{r}).
\end{align*}
Set $z_r(t) = \int \psi_r|u(t)|^2 dx$. Then from the density by $H^2$ data and continuous dependency of solutions it follows that
\begin{align}\label{dilation}
\frac{d}{dt}z_r = 2 \textrm{Im}\int \nabla \psi_r \cdot \nabla u \bar{u}~dx
\end{align}
and
\begin{align}\begin{aligned}\label{lvirial}
\frac{d^2}{dt^2}z_r &= 2 \textrm{Im}\int \left[-\Delta \psi_r u_t \bar{u}-\left(\nabla \psi_r \cdot \nabla\bar{u}\right)u_t +\left(\nabla \psi_r \cdot \nabla u \right)\bar{u_t}\right] dx\\
&= 4\textrm{Re}\int(\nabla^2 \psi_r \cdot \nabla\bar{u})\nabla u dx  -\frac{4-2b}{3-b}  \int(\Delta \psi_r)g|u|^{6-2b}dx\\
&\qquad + \frac{2}{3-b}\int\left(\nabla \psi_r \cdot \nabla g\right)|u|^{6-2b} dx - \int (\Delta^2\psi_r)|u|^2 dx.
\end{aligned}\end{align}
Note that \eqref{lvirial} has been obtained without radial symmetry. By integrating and taking limit $r\to \infty$ on both sides of \eqref{dilation} and \eqref{lvirial}, Fatou's lemma yields
\begin{align*}
\int |x|^2|u(t)|^2\,dx &\le 8\int_0^t\int_0^s\int\left(|\nabla u(t')|^2 - (1-k_g)g|u(t')|^{p+1}\right) dx dt'ds\\
& \qquad \quad + 2 t{\rm Im}\int (\nabla \varphi \cdot x)\varphi\,dx + \int |x|^2|\varphi|^2\,dx.
\end{align*}
Then from Corollary \ref{blowup-cor} it follows that
$$
\int |x|^2|u(t)|^2\,dx \le - C_g\bar{\delta} t^2 + 2 t{\rm Im}\int (\nabla \varphi \cdot x)\varphi\,dx + \int |x|^2|\varphi|^2\,dx
$$
for some constant $C_g$. The last inequality gives us that the maximal interval is bounded.

For the part $(2)$, we need another $\psi_r$. Let us introduce the function $\widetilde{\psi} \in C^4([0, \infty))$ such that $\widetilde{\psi}(s) = s$ for $0 \le s \le 1$, smooth for $1 < s < 10$, and $0$ for $s \ge 10$ and further that
$0 \le \widetilde{\psi} \le 1$ and $\widetilde{\psi}'(s) \le 1$ for all $s \ge 0$. For the construction of such function see Appendix B of \cite{bole}.

Now let $\widetilde{\psi}_r(s) = r \widetilde{\psi}(\frac sr)$ and $b_r(|x|) = \int_0^{|x|}\widetilde{\psi}_r(s)\,ds$. Then, by \eqref{lvirial} and radial symmetry of $u$, we get
\begin{align}\begin{aligned}\label{lvirial2}
\frac{d^2}{dt^2}z_r &= 4\int \widetilde{\psi}_r'(|x|)|\nabla u|^2\,dx - \frac{2p-2}{p+1}\int \nabla \cdot \left(\frac{x}{|x|}\widetilde{\psi}_r(|x|)\right)g|u|^{6-2b}\,dx\\
&\qquad  + \frac2{3-b}\int \frac{\widetilde{\psi}_r(|x|)}{|x|}(x\cdot\nabla g) |u|^{6-2b}\,dx - \int \Delta \nabla\cdot\left(\frac{x}{|x|}\widetilde{\psi}_r(|x|)\right)|u|^2\,dx.
\end{aligned}\end{align}
Since $\widetilde{\psi}_r'(s) \le 1$ and $x\cdot \nabla g \le (p+1)(k_g - \rho)g$, we then have
\begin{align*}
\frac{d^2}{dt^2} z_r(t) &\le 4{\int}\left(|\nabla u(t)|^2 - (1-k_g+\rho)g|u(t)|^{6-2b}\right) dx\\
&\qquad\qquad + 4 \int\left[1 - \frac{{2-b}}{6-2b}\nabla \cdot \left(\frac{x}{|x|}\widetilde{\psi}_r(|x|)\right)\right]g|u|^{6-2b} - \int \Delta \nabla\cdot\left(\frac{x}{|x|}\widetilde{\psi}_r(|x|)\right) |u|^2dx\\
&\le 4{\int}\left(|\nabla u(t)|^2 - (1-k_g+\rho)g|u(t)|^{6-2b}\right) dx + C\left[\underset{|x| \ge r}{\int} g|u|^{6-2b} + \frac{|u|^2}{|x|^2}dx \right]\\
&\le 4{\int}\left(|\nabla u(t)|^2 - (1-k_g+\rho)g|u(t)|^{6-2b}\right) dx + Cg_s\||x|^{-\frac{b}2}u\|_{L_x^{\frac{2}{b-1}}(|x|\ge r)}^2\|u\|_{L_x^2}^{4-2b}\\
&\qquad\qquad\qquad\qquad\qquad\qquad\qquad\qquad\qquad\quad\;\; + Cr^{-2}\|u\|_{L_x^2}^2.
\end{align*}
To control the second term we use the Lemma \ref{emb-w}:
$$
\||x|^{-\frac{7-3p}{4}}u\|_{L_x^{\frac4{3-p}}} \le C_0 \|\nabla u\|_{L_x^2}.
$$
The mass conservation (Remark \ref{m-e-cons}) gives us that
\begin{align*}
\frac{d^2}{dt^2} z_r(t) \le 4(1+\varepsilon(r))\int\left(|\nabla u(t)|^2 - \frac{(1-k_g+\rho)}{1+\varepsilon(r)}g|u(t)|^{p+1}\right) dx + Cr^{-2}\|\varphi\|_{L_x^2}^2,
\end{align*}
where $\varepsilon(r) =\frac14 CC_0^2g_s\|\varphi\|_{L_x^2}^{p-1}r^{-\frac{p-1}2}$. Hence if we choose $r$ large enough, then since $\rho > 0$, by Lemma \ref{blowup-cor} we deduce that
\begin{align*}
\frac{d^2}{dt^2} z_r(t) \;\le\; -\frac{C_g\bar{\delta}}2.
\end{align*}
By the same argument as of part $(1)$ we obtain the desired result.

\section*{Acknowledgements}
This work was supported by National University Promotion Development Project in 2019.


\end{document}